\documentclass[reqno,10pt]{amsart}

\usepackage{a4wide}
\usepackage{mathrsfs}
\usepackage{mathtools}
\usepackage{amsmath}
\usepackage{amssymb}
\usepackage{bbm}
\numberwithin{equation}{section}
\usepackage[colorlinks,citecolor=green,linkcolor=red]{hyperref}

\usepackage[latin1]{inputenc}

\newcommand{\N}{\mathbb{N}}
\newcommand{\Q}{\mathbb{Q}}

\newcommand{\sfd}{{\sf d}}
\renewcommand{\d}{{\mathrm d}}
\newcommand{\e}{{\sf e}}
\newcommand{\X}{{\rm X}}

\newcommand{\mm}{\mathfrak{m}}
\newcommand{\1}{\mathbbm 1}
\newcommand{\ppi}{{\mbox{\boldmath\(\pi\)}}}
\newcommand{\sppi}{{\mbox{\scriptsize\boldmath\(\pi\)}}}

\newcommand{\fr}{\penalty-20\null\hfill\(\blacksquare\)}

\newtheorem{theorem}{Theorem}

\newtheorem{lemma}[theorem]{Lemma}
\newtheorem{proposition}[theorem]{Proposition}

\newtheorem{remark}[theorem]{Remark}

\title{A short proof of the existence of master test plans}
\author{Enrico Pasqualetto}
\address[Enrico Pasqualetto]{Scuola Normale Superiore, Piazza dei Cavalieri 7, 56126 Pisa, Italy}
\email{enrico.pasqualetto@sns.it}
\date{\today}
\begin{document}
\begin{abstract}
The aim of this brief note is to provide a quick and elementary proof of the following known fact: on a metric measure space whose
Sobolev space is separable, there exists a test plan that is sufficient to identify the minimal weak upper gradient of every Sobolev function.
\end{abstract}
\maketitle
\section*{Introduction}
In recent years, several (mostly equivalent) notions of Sobolev space on a metric measure space were studied.
One approach is via weak upper gradients: given a metric measure space \((\X,\sfd,\mm)\) and an exponent \(p\in(1,\infty)\),
one says that \(G\in L^p(\mm)\) is a \(p\)-weak upper gradient of \(f\in L^p(\mm)\) if for \emph{almost every}
absolutely continuous curve \(\gamma\colon[0,1]\to\X\) it holds that \(f\circ\gamma\in W^{1,1}(0,1)\) and
\begin{equation}\tag{\(\star\)}\label{eq:intro_wug}
|(f\circ\gamma)'_t|\leq G(\gamma_t)|\dot\gamma_t|,\quad\text{ for }\mathcal L^1\text{-a.e.\ }t\in[0,1].
\end{equation}
The Sobolev space \(W^{1,p}(\X)\) consists of those functions \(f\in L^p(\mm)\) admitting a \(p\)-weak upper gradient.
The \(\mm\)-a.e.\ minimal \(p\)-weak upper gradient of \(f\in W^{1,p}(\X)\) is usually denoted by \(|Df|_p\).

There are various different ways to quantify the exceptional curves in the weak upper gradient condition \eqref{eq:intro_wug}.
Shanmugalingam introduced in \cite{Shanmugalingam00} the Newtonian--Sobolev space, where \eqref{eq:intro_wug} is required
to hold along \({\rm Mod}_p\)-a.e.\ curve; here, \({\rm Mod}_p\) denotes the \(p\)-modulus, which is an outer measure on paths.
An alternative strategy, proposed by Ambrosio--Gigli--Savar\'{e} in \cite{AmbrosioGigliSavare11-3} after \cite{AmbrosioGigliSavare11},
is based on \(q\)-test plans (where \(q\) is the conjugate exponent of \(p\)): these are Borel probability measures on the space of
curves, having bounded compression and finite kinetic \(q\)-energy; see \eqref{eq:test_plan}. The approach of \cite{AmbrosioGigliSavare11-3}
requires that, for any \(q\)-test plan \(\ppi\), the property \eqref{eq:intro_wug} holds for \(\ppi\)-a.e.\ \(\gamma\).

While the \(p\)-modulus is a non-Borel outer measure, \(q\)-test plans are \(\sigma\)-additive Borel measures, but a priori
uncountably many of them are needed to recover minimal \(p\)-weak upper gradients. Nevertheless, it is shown in \cite{Pas20}
that a single \(q\)-test plan (called a master \(q\)-test plan) is sufficient to compute the minimal \(p\)-weak upper gradient
of each Sobolev function. Later on, this result was refined further: as proven in \cite{GigliNobili21} via optimal transport techniques,
it is possible to find a master \(q\)-test plan that is also capable of detecting which functions are Sobolev. Analogous results for
bounded variation functions (corresponding to the borderline case \(p=1\)) were obtained in \cite{NPS21}.
\medskip

The aim of the present note is to provide a rather elementary and direct proof of the existence of a master \(q\)-test plan on a
metric measure space \((\X,\sfd,\mm)\) whose Sobolev space \(W^{1,p}(\X)\) is separable. The argument is very general and
ultimately relies on basic measure-theoretical tools; in particular, it does not require the existence of any test plan having
special properties. The separability of \(W^{1,p}(\X)\) is a quite mild assumption: it is in force, for instance, whenever
\(W^{1,p}(\X)\) is reflexive (cf.\ with \cite[Proposition 42]{ACM14}).
In turn, the reflexivity of the Sobolev space is satisfied in most cases of interest: on metrically doubling spaces \cite[Corollary 41]{ACM14};
on spaces admitting a \(p\)-weak differentiable structure \cite[Corollary 6.7]{ErikssonBiqueSoultanis21}; on infinitesimally Hilbertian spaces
\cite{Gigli12} (in the case where \(p=2\)), such as \({\sf RCD}(K,\infty)\) spaces \cite{AmbrosioGigliSavare11-2}, as well as
Hilbert spaces (or, more generally, locally \({\sf CAT}(\kappa)\) spaces) endowed with an arbitrary boundedly-finite Borel measure \cite{DMGPS18}.
However, metric measure spaces having non-separable (and thus non-reflexive) Sobolev space do actually exist: one such example was constructed
in \cite[Proposition 44]{ACM14} after \cite[Section 12.5]{Heinonen07}.
\medskip

Let us briefly sketch the strategy we adopt to obtain a master \(q\)-test plan on any metric measure space having separable \(p\)-Sobolev space.
Given a function \(f\in W^{1,p}(\X)\) and a \(q\)-test plan \(\ppi\) on \(\X\), the key object we consider is the \(\mm\)-a.e.\ defined function
\(\Phi[\ppi,f]\colon\X\to[0,+\infty)\), which is given by
\[
\Phi[\ppi,f](x)\coloneqq\underset{\substack{\hat\sppi_x\text{-a.e.\ }(\gamma,t) \\ \text{s.t.\ }|\dot\gamma_t|>0}}{\rm ess\,sup}
\bigg|\frac{(f\circ\gamma)'_t}{|\dot\gamma_t|}\bigg|\leq|Df|_p(x),\quad\text{ for }\mm\text{-a.e.\ }x\in\X,
\]
where \(\{\hat\ppi_x\}_{x\in\X}\) is the disintegration of \(\hat\ppi\coloneqq\ppi\otimes\mathcal L^1|_{[0,1]}\) along
the evaluation map \(\e(\gamma,t)\coloneqq\gamma_t\). Roughly speaking, \(\Phi[\ppi,f]\) is measuring how far the plan
\(\ppi\) is from saturating the minimal weak upper gradient inequality for \(f\). Up to some small technical differences,
the function \(\Phi[\ppi,f]\) was previously considered in \cite{ErikssonBiqueSoultanis21}. An important observation is
that a family \(\Pi\) of \(q\)-test plans is a master family for some set \(\mathcal S\subseteq W^{1,p}(\X)\) of Sobolev
functions exactly when \(\bigvee_{\sppi\in\Pi}\Phi[\ppi,f]=|Df|_p\) for every \(f\in\mathcal S\); see Proposition \ref{prop:equiv_master_family}.
Since each essential supremum of measurable functions can be expressed \(\mm\)-a.e.\ as a countable supremum, if \(\mathcal S\)
a countable dense subset of \(W^{1,p}(\X)\), then one can find a countable master family \(\Pi\) of \(q\)-test plans for \(\mathcal S\).
The countable family \(\Pi\) can be easily reduced to a single master \(q\)-test plan \(\ppi\) for \(\mathcal S\), see Remark
\ref{rmk:equiv_contable_master_tp}. Finally, a simple continuity argument (namely, Lemma \ref{lem:cont_Phi}) allows to conclude
that \(\ppi\) is a master \(q\)-test plan for the whole Sobolev space \(W^{1,p}(\X)\); see Theorem \ref{thm:ex_master_tp}.
\medskip

\noindent\textbf{Acknowledgements.}
The author was supported by the Balzan project led by Luigi Ambrosio.
\section{Preliminaries}
\subsection{General terminology}
By a \textbf{metric measure space} \((\X,\sfd,\mm)\) we mean a complete and separable metric space \((\X,\sfd)\)
with a boundedly-finite Borel measure \(\mm\geq 0\). The space \(C([0,1],\X)\) of continuous curves
is complete and separable if endowed with the supremum distance, which is given by
\(\sfd_{\sup}(\gamma,\sigma)\coloneqq\max_{t\in[0,1]}\sfd(\gamma_t,\sigma_t)\).
A curve \(\gamma\in C([0,1],\X)\) is \textbf{absolutely continuous} if
\[
{\sf ms}(\gamma,t)=|\dot\gamma_t|\coloneqq\lim_{h\to 0}\frac{\sfd(\gamma_{t+h},\gamma_t)}{|h|}
\]
exists for \(\mathcal L_1\)-a.e.\ \(t\in[0,1]\), the resulting function \(|\dot\gamma|\) belongs to \(L^1(0,1)\),
and \(\sfd(\gamma_t,\gamma_s)\leq\int_s^t|\dot\gamma_r|\,\d r\) holds for every \(s,t\in[0,1]\) with \(s\leq t\);
by \(\mathcal L_1\) we mean the restriction of the Lebesgue measure \(\mathcal L^1\) to the unit interval \([0,1]\).
Given any \(t\in[0,1]\), we denote by \(\e_t\colon C([0,1],\X)\to\X\) the \textbf{evaluation map at time \(t\)},
which is defined as \(\e_t(\gamma)\coloneqq\gamma_t\) for every curve \(\gamma\in C([0,1],\X)\). Moreover, we will consider the
(joint) \textbf{evaluation map} \(\e\colon C([0,1],\X)\times[0,1]\to\X\), given by \(\e(\gamma,t)\coloneqq\e_t(\gamma)=\gamma_t\)
for every \((\gamma,t)\in C([0,1],\X)\times[0,1]\). Observe that each map \(\e_t\) is \(1\)-Lipschitz, while \(\e\) is continuous.
\medskip

In this note, we deal with two different concepts of essential supremum. Firstly, given a metric measure space \((\X,\sfd,\mm)\)
and a Borel function \(f\colon\X\to[0,+\infty)\), we define \({\rm ess\,sup}_\mm f\in[0,+\infty]\) as
\[
\underset{\mm}{\rm ess\,sup}\,f\coloneqq\inf\Big\{\lambda\in[0,+\infty)\;\Big|\;f\leq\lambda\,\text{ holds }\mm\text{-a.e.}\Big\}.
\]
Secondly, given a (possibly uncountable) family \(\{f_i\}_{i\in I}\) of Borel functions \(f_i\colon\X\to[0,+\infty)\), we define
\(\bigvee_{i\in I}f_i\) as the \(\mm\)-a.e.\ unique Borel function \(f\colon\X\to[0,+\infty]\) such that \(f\geq f_i\) holds
\(\mm\)-a.e.\ for all \(i\in I\), and satisfying \(f\leq g\) in the \(\mm\)-a.e.\ sense for any \(g\colon\X\to[0,+\infty]\)
Borel with the same property as \(f\). One can find \(\mathcal C\subseteq I\) countable such that
\(\bigvee_{i\in I}f_i(x)=\sup_{i\in\mathcal C}f_i(x)\) for \(\mm\)-a.e.\ \(x\in\X\).
\subsection{Test plans and Sobolev spaces}
Fix exponents \(p,q\in(1,\infty)\) with \(\frac{1}{p}+\frac{1}{q}=1\). As in \cite[Definition 4.6]{AmbrosioGigliSavare11-3},
by a \textbf{\(q\)-test plan} on \(\X\) we mean a Borel probability measure \(\ppi\in\mathscr P\big(C([0,1],\X)\big)\)
concentrated on absolutely continuous curves that satisfies the following conditions for some \(C>0\):
\begin{equation}\label{eq:test_plan}
(\e_t)_\#\ppi\leq C\mm,\;\;\;\text{for every }t\in[0,1],\qquad K_q(\ppi)\coloneqq\int\!\!\!\int_0^1|\dot\gamma_t|^q\,\d t\,\d\ppi(\gamma)<+\infty.
\end{equation}
We denote by \(C(\ppi)\) the smallest such \(C\). For any \(q\)-test plan \(\ppi\) on \(\X\), we adopt this notation:
\[
\hat\ppi\coloneqq\ppi\otimes\mathcal L_1\in\mathscr P\big(C([0,1],\X)\times[0,1]\big).
\]
Given a family \(\Pi\) of \(q\)-test plan on \(\X\) and functions \(f,G\in L^p(\mm)\), we say that \(G\) is a \textbf{\((\Pi,p)\)-weak upper gradient}
of \(f\) provided for any \(\ppi\in\Pi\) it holds that \(f\circ\gamma\in W^{1,1}(0,1)\) for \(\ppi\)-a.e.\ \(\gamma\) and
\begin{equation}\label{eq:def_wug}
|(f\circ\gamma)'_t|\leq G(\gamma_t)|\dot\gamma_t|,\quad\text{ for }\hat\ppi\text{-a.e.\ }(\gamma,t)\in C([0,1],\X)\times[0,1].
\end{equation}
If \(f\in L^p(\mm)\) admits a \((\Pi,p)\)-weak upper gradient, then we declare that \(f\in W^{1,p}_\Pi(\X)\). The previous definitions
are taken from \cite[Definition 1.11]{Pas20}. Following \cite[Definition 1.13]{Pas20}, to any given function \(f\in W^{1,p}_\Pi(\X)\) we associate its
\textbf{minimal \((\Pi,p)\)-weak upper gradient} \(|Df|_{\Pi,p}\in L^p(\mm)\), which is defined as the essential infimum of all the \((\Pi,p)\)-weak
upper gradients of \(f\). Thanks to \cite[Lemma 1.12]{Pas20}, we know that \(|Df|_{\Pi,p}\) itself satisfies \eqref{eq:def_wug}. When \(\Pi\) is
the collection \(\Pi_q(\X)\) of all \(q\)-test plans, one recovers the notion of \textbf{\(p\)-Sobolev space} \(W^{1,p}(\X)=W^{1,p}_{\Pi_q}(\X)\) from
\cite{AmbrosioGigliSavare11-3}. For any \(f\in W^{1,p}(\X)\), we write \(|Df|_p\) instead of \(|Df|_{\Pi_q,p}\) and we call it just the
\textbf{minimal \(p\)-weak upper gradient} of \(f\). Recall that \(W^{1,p}(\X)\) is a Banach space if endowed with the norm
\[
\|f\|_{W^{1,p}(\X)}\coloneqq\bigg(\int|f|^p\,\d\mm+\int|Df|_p^p\,\d\mm\bigg)^{1/p},\quad\text{ for every }f\in W^{1,p}(\X).
\]
Notice that \(W^{1,p}(\X)\subseteq W^{1,p}_\Pi(\X)\) holds for any family \(\Pi\) of \(q\)-test plans on \(\X\) and \(|Df|_{\Pi,p}\leq|Df|_p\)
in the \(\mm\)-a.e.\ sense for every \(f\in W^{1,p}(\X)\). In analogy with \cite[Definition 2.5]{Pas20}, we say that a given family \(\Pi\) of
\(q\)-test plans on \(\X\) is an \textbf{\(\mathcal S\)-master family of \(q\)-test plans}, for some \(\mathcal S\subseteq W^{1,p}(\X)\), if
\[
|Df|_{\Pi,p}=|Df|_p,\;\;\;\text{holds }\mm\text{-a.e.\ on }\X,\quad\text{ for every }f\in\mathcal S.
\]
When \(\mathcal S=W^{1,p}(\X)\), we just say that \(\Pi\) is a \textbf{master family of \(q\)-test plans}. Notice that we are not requiring that
a master family \(\Pi\) of \(q\)-test plans verifies \(W^{1,p}_\Pi(\X)=W^{1,p}(\X)\). For brevity, we say that a \(q\)-test plan \(\ppi\)
is an \textbf{\(\mathcal S\)-master \(q\)-test plan} if \(\{\ppi\}\) is an \(\mathcal S\)-master family of \(q\)-test plans.
\begin{remark}\label{rmk:equiv_contable_master_tp}{\rm
Suppose to have a countable \(\mathcal S\)-master family \(\Pi=\{\ppi^n\}_{n\in\N}\) of \(q\)-test plans on \(\X\). Set
\[
\ppi\coloneqq\frac{\tilde\ppi}{\tilde\ppi\big(C([0,1],\X)\big)},\quad\text{ where }
\tilde\ppi\coloneqq\sum_{n=1}^\infty\frac{\ppi^n}{2^n\max\big\{C(\ppi^n),K_q(\ppi^n),1\big\}}.
\]
Then \(\ppi\) is an \(\mathcal S\)-master \(q\)-test plan on \(\X\), cf.\ with \textsc{Step 3} in the proof of \cite[Theorem 2.6]{Pas20}.
\fr}\end{remark}
\section{Main result}
Let \((\X,\sfd,\mm)\) be a metric measure space. Let \(p,q\in(1,\infty)\) satisfy \(\frac{1}{p}+\frac{1}{q}=1\). Given a \(q\)-test plan
\(\ppi\) on \(\X\) and \(f\in W^{1,p}(\X)\), we define the auxiliary functions \(\phi_+[\ppi,f],\phi_-[\ppi,f]\in L^p(\hat\ppi)\) as
\[
\phi_\pm[\ppi,f](\gamma,t)\coloneqq\pm\1_{\{{\sf ms}>0\}}(\gamma,t)\frac{(f\circ\gamma)'_t}{|\dot\gamma_t|},
\quad\text{ for }\hat\ppi\text{-a.e.\ }(\gamma,t).
\]
Moreover, we define \(\mm\)-a.e.\ the functions \(\Phi_+[\ppi,f],\Phi_-[\ppi,f]\in L^p(\mm)\) as
\[
\Phi_\pm[\ppi,f](x)\coloneqq\underset{\hat\sppi_x}{\rm ess\,sup}\,\phi_\pm[\ppi,f],\quad
\text{ for }\mm\text{-a.e.\ }x\in S_\sppi\coloneqq\bigg\{\frac{\d(\e_\#\hat\ppi)}{\d\mm}>0\bigg\},
\]
and \(\Phi_\pm[\ppi,f](x)\coloneqq 0\) for \(\mm\)-a.e.\ \(x\in\X\setminus S_\sppi\), where \(\hat\ppi=\int\hat\ppi_x\,\d(\e_\#\hat\ppi)(x)\)
is the disintegration of \(\hat\ppi\) along \(\e\). Setting
\(B_\lambda^\pm\coloneqq\big\{x\in\X\,\big|\,\hat\ppi_x\big(\{\phi_\pm[\ppi,f]>\lambda\}\big)=0\big\}\) for every \(\lambda>0\),
we can express \(\mm\)-a.e.\ the function \(\Phi_\pm[\ppi,f]\) as  \(\inf_{\lambda\in\Q\cap(0,+\infty)}\lambda\1_{B_\lambda^\pm}\),
showing the measurability of \(\Phi_\pm[\ppi,f]\). Finally, we define \(\mm\)-a.e.\ the function \(\Phi[\ppi,f]\in L^p(\mm)\)
as \(\Phi[\ppi,f]\coloneqq\max\big\{\Phi_+[\ppi,f],\Phi_-[\ppi,f]\big\}\leq|Df|_p\).
\begin{proposition}\label{prop:equiv_master_family}
Let \((\X,\sfd,\mm)\) be a metric measure space. Let \(p,q\in(1,\infty)\) be such that \(\frac{1}{p}+\frac{1}{q}=1\).
Let \(\mathcal S\) be a subset of \(W^{1,p}(\X)\) and let \(\Pi\) be a family of \(q\)-test plans on \(\X\). Then \(\Pi\)
is an \(\mathcal S\)-master family of \(q\)-test plans if and only if
\begin{equation}\label{eq:equiv_master_family_claim}
\bigvee_{\sppi\in\Pi}\Phi[\ppi,f]=|Df|_p,\;\;\;\text{in the }\mm\text{-a.e.\ sense,}\quad\text{ for every }f\in\mathcal S.
\end{equation}
\end{proposition}
\begin{proof}
To prove necessity, assume \(\Pi\) is an \(\mathcal S\)-master family of \(q\)-test plans. We argue by contradiction:
suppose \eqref{eq:equiv_master_family_claim} fails. Then there exist \(f\in\mathcal S\), \(\varepsilon>0\),
and a Borel set \(P\subseteq\X\) with \(\mm(P)>0\) such that for any \(\ppi\in\Pi\) the inequalities
\(\Phi_\pm[\ppi,f]<|Df|_p-\varepsilon\) hold \(\mm\)-a.e.\ on \(P\). Hence, for any \(\ppi\in\Pi\)
we have that for \((\e_\#\hat\ppi)\)-a.e.\ \(x\in P\) and \(\hat\ppi_x\text{-a.e.\ }(\gamma,t)\) it holds that
\(|(f\circ\gamma)'_t|\leq\big(|Df|_p(\gamma_t)-\varepsilon\big)|\dot\gamma_t|\).
This implies that for any \(\ppi\in\Pi\) it holds \(|(f\circ\gamma)'_t|\leq\big(|Df|_p(\gamma_t)-\varepsilon\big)|\dot\gamma_t|\)
for \(\hat\ppi\)-a.e.\ \((\gamma,t)\in\e^{-1}(P)\), thus the function \(G\coloneqq\1_P\big(|Df|_p-\varepsilon\big)+\1_{\X\setminus P}|Df|_p\in L^p(\mm)\)
is a \((\Pi,p)\)-weak upper gradient of \(f\). Given that the strict inequality \(G<|Df|_p\) holds on a positive \(\mm\)-measure set, this contradicts
the assumption that \(\Pi\) is an \(\mathcal S\)-master family of \(q\)-test plans. Consequently, property \eqref{eq:equiv_master_family_claim} is proven.

To prove sufficiency, assume \eqref{eq:equiv_master_family_claim} holds. To prove that \(\Pi\) is
an \(\mathcal S\)-master family of \(q\)-test plans amounts to showing that if \(G\) is a \((\Pi,p)\)-weak upper gradient of \(f\in\mathcal S\),
then \(|Df|_p\leq G\) holds \(\mm\)-a.e. Given such \(f\) and \(G\), we know from \eqref{eq:equiv_master_family_claim} that for any \(k\in\N\)
there exist \((\ppi^{n,k})_n\subseteq\Pi\) and a Borel partition \(\{A_{n,k}^+,A_{n,k}^-\}_{n\in\N}\) of \(\X\) such that
\(\Phi_\pm[\ppi^{n,k},f]>|Df|_p-1/k\) holds \(\mm\)-a.e.\ on \(A_{n,k}^\pm\). Hence, given any \(n\in\N\),
for \(\mm\)-a.e.\ \(x\in A_{n,k}^\pm\) we have that \(x\in S_{\sppi^{n,k}}\) and that \(\hat\ppi^{n,k}_x(B^{n,k}_{\pm,x})>0\), where
we define \(B^{n,k}_{\pm,x}\coloneqq\big\{\phi_\pm[\ppi^{n,k},f]>|Df|_p-1/k\big\}\). Then for \(\mm\)-a.e.\ point \(x\in A_{n,k}^\pm\) we have
\[
|Df|_p(\gamma_t)\leq\pm\1_{\{{\sf ms}>0\}}(\gamma,t)\,\frac{(f\circ\gamma)'_t}{|\dot\gamma_t|}+\frac{1}{k}
\leq G(\gamma_t)+\frac{1}{k},\quad\text{ for }\hat\ppi^{n,k}_x\text{-a.e.\ }(\gamma,t)\in B^{n,k}_{\pm,x}.
\]
Hence, for \(\mm\)-a.e.\ \(x\in A_{n,k}^+\cup A_{n,k}^-\) it holds \(\hat\ppi^{n,k}_x\big(\big\{|Df|_p\circ\e\leq G\circ\e+1/k\big\}\big)>0\),
whence it follows that \(|Df|_p\leq G+1/k\) holds \(\mm\)-a.e.\ on \(A_{n,k}^+\cup A_{n,k}^-\). Therefore, \(|Df|_p\leq G\) holds \(\mm\)-a.e.\ on \(\X\).
\end{proof}
\begin{lemma}\label{lem:cont_Phi}
Let \((\X,\sfd,\mm)\) be a metric measure space. Let \(p,q\in(1,\infty)\) satisfy \(\frac{1}{p}+\frac{1}{q}=1\).
Fix two functions \(f,g\in W^{1,p}(\X)\) and a \(q\)-test plan \(\ppi\) on \(\X\). Then
\(\big|\Phi[\ppi,f]-\Phi[\ppi,g]\big|\leq|D(f-g)|_p\) in the \(\mm\)-a.e.\ sense.
In particular, if \((f_n)_n\subseteq W^{1,p}(\X)\) satisfies \(|D(f-f_n)|_p\to 0\) pointwise \(\mm\)-a.e., then
\begin{equation}\label{eq:cont_Phi_claim2}
\bigvee_{\sppi\in\Pi}\Phi[\ppi,f]=\lim_{n\to\infty}\bigvee_{\sppi\in\Pi}\Phi[\ppi,f_n],\;\;\;\mm\text{-a.e.\ on }\X,
\quad\text{ for every }\Pi\subseteq\Pi_q(\X).
\end{equation}
\end{lemma}
\begin{proof}
We know from the definition of \(\Phi\) that for \(\mm\)-a.e.\ \(x\in S_\sppi\) and for \(\hat\ppi_x\)-a.e.\ \((\gamma,t)\) it holds that
\[
\big|\phi_\pm[\ppi,f](\gamma,t)-\phi_\pm[\ppi,g](\gamma,t)\big|=
\1_{\{{\sf ms}>0\}}(\gamma,t)\frac{\big|\big((f-g)\circ\gamma\big)'_t\big|}{|\dot\gamma_t|}\leq|D(f-g)|_p(x).
\]
By passing to the essential supremum with respect to \(\hat\ppi_x\), we obtain for \(\mm\)-a.e.\ \(x\in S_\sppi\) that
\[
\Phi[\ppi,f](x)\leq\Phi[\ppi,g](x)+|D(f-g)|_p(x),\qquad\Phi[\ppi,g](x)\leq\Phi[\ppi,f](x)+|D(f-g)|_p(x),
\]
whence the first claim follows. To prove \eqref{eq:cont_Phi_claim2}, define \(I\coloneqq\bigvee_{\sppi\in\Pi}\Phi[\ppi,f]\) and
\(I_n\coloneqq\bigvee_{\sppi\in\Pi}\Phi[\ppi,f_n]\). It follows from the first claim that if \(n\in\N\) and \(\ppi\in\Pi\), then
\(\Phi[\ppi,f]-|D(f-f_n)|_p\leq\Phi[\ppi,f_n]\leq I_n\) and \(\Phi[\ppi,f_n]-|D(f-f_n)|_p\leq\Phi[\ppi,f]\leq I\) in the \(\mm\)-a.e.\ sense.
By the arbitrariness of \(\pi\in\Pi\), we deduce that \(I-|D(f-f_n)|_p\leq I_n\) and \(I_n-|D(f-f_n)|_p\leq I\) hold \(\mm\)-a.e.\ on \(\X\).
We have shown that \(|I-I_n|\leq |D(f-f_n)|_p\) holds \(\mm\)-a.e.\ for all \(n\in\N\), thus by letting \(n\to\infty\) we obtain \eqref{eq:cont_Phi_claim2}.
\end{proof}
\begin{theorem}\label{thm:ex_master_tp}
Let \((\X,\sfd,\mm)\) be a metric measure space. Let \(p,q\in(1,\infty)\) satisfy \(\frac{1}{p}+\frac{1}{q}=1\).
Let \(\mathcal S\) be a separable subset of \(W^{1,p}(\X)\). Then there exists an \(\mathcal S\)-master \(q\)-test plan on \(\X\).
In particular, if the Sobolev space \(W^{1,p}(\X)\) is separable, then there exists a master \(q\)-test plan on \(\X\).
\end{theorem}
\begin{proof}
Fix a dense sequence \((f_n)_n\subseteq\mathcal S\). Being \(\Pi_q(\X)\) a master family of \(q\)-test plans, Proposition \ref{prop:equiv_master_family}
ensures that \(\bigvee_{\sppi\in\Pi_q(\X)}\Phi[\ppi,f_n]=|Df_n|_p\) holds \(\mm\)-a.e.\ for any \(n\in\N\). In particular, we can find a countable set
\(\Pi^n\subseteq\Pi_q(\X)\) such that \(\sup_{\sppi\in\Pi^n}\Phi[\ppi,f_n]=|Df_n|_p\) holds \(\mm\)-a.e. Consider now the countable family
\(\Pi\coloneqq\bigcup_{n\in\N}\Pi^n\). It holds that
\begin{equation}\label{eq:ex_master_tp_aux}
\bigvee_{\sppi\in\Pi}\Phi[\ppi,f_n]\geq\sup_{\sppi\in\Pi^n}\Phi[\ppi,f_n]=|Df_n|_p,\;\;\;\text{in the }\mm\text{-a.e.\ sense,}\quad\text{ for every }n\in\N.
\end{equation}
Fix any \(f\in\mathcal S\). We can find \((n_k)_k\subseteq\N\) such that \(f_{n_k}\to f\) in \(W^{1,p}(\X)\),
thus (up to a non-relabelled subsequence) we have \(|D(f_{n_k}-f)|_p\to 0\) pointwise \(\mm\)-a.e. By combining \eqref{eq:ex_master_tp_aux}
with Lemma \ref{lem:cont_Phi}, we deduce that \(\bigvee_{\sppi\in\Pi}\Phi[\ppi,f]=|Df|_p\) in the \(\mm\)-a.e.\ sense. Being \(f\in\mathcal S\)
arbitrary, we have proven that \(\Pi\) satisfies \eqref{eq:equiv_master_family_claim}, thus Proposition \ref{prop:equiv_master_family} ensures
that \(\Pi\) is an \(\mathcal S\)-master family of \(q\)-test plans. By recalling Remark \ref{rmk:equiv_contable_master_tp}, we finally conclude
that an \(\mathcal S\)-master \(q\)-test plan on \(\X\) exists.
\end{proof}
\begin{remark}{\rm
Given a metric measure space \((\X,\sfd,\mm)\) and exponents \(p,q\in(1,\infty)\) with \(\frac{1}{p}+\frac{1}{q}=1\), we know
from \cite[Theorem 2.6]{Pas20} (or \cite[Theorem A.2]{GigliNobili21}) that a master \(q\)-test plan \(\ppi\) on \(\X\) exists;
notice that here no separability assumption on \(W^{1,p}(\X)\) is made. Then Proposition \ref{prop:equiv_master_family} implies that
\begin{equation}\label{eq:repr_mwug}
\Phi[\ppi,f]=|Df|_p,\;\;\;\text{holds }\mm\text{-a.e.\ on }\X,\quad\text{ for every }f\in W^{1,p}(\X).
\end{equation}
This improves a result of \cite{ErikssonBiqueSoultanis21}: in \cite[Theorem 1.1]{ErikssonBiqueSoultanis21} it is shown that
for any \(p\)-Sobolev function \(f\) there exists a plan \(\ppi_f\) that `represents' its minimal weak upper gradient, meaning
that \(\Phi[\ppi_f,f]=|Df|_p\), while \eqref{eq:repr_mwug} says that a single \(q\)-test plan \(\ppi\) represents \(|Df|_p\)
for \emph{every} \(f\in W^{1,p}(\X)\).
\fr}\end{remark}
\def\cprime{$'$} \def\cprime{$'$}


\begin{thebibliography}{10}

\bibitem{ACM14}
{\sc L.~Ambrosio, M.~Colombo, and S.~Di~Marino}, {\em Sobolev spaces in metric
  measure spaces: reflexivity and lower semicontinuity of slope}, in
  Variational Methods for Evolving Objects, Math. Soc. Japan, 2015, pp.~1--58.

\bibitem{AmbrosioGigliSavare11-3}
{\sc L.~Ambrosio, N.~Gigli, and G.~Savar{\'e}}, {\em Density of {L}ipschitz
  functions and equivalence of weak gradients in metric measure spaces}, Rev.
  Mat. Iberoam., 29 (2013), pp.~969--996.

\bibitem{AmbrosioGigliSavare11}
\leavevmode\vrule height 2pt depth -1.6pt width 23pt, {\em Calculus and heat
  flow in metric measure spaces and applications to spaces with {R}icci bounds
  from below}, Invent. Math., 195 (2014), pp.~289--391.

\bibitem{AmbrosioGigliSavare11-2}
\leavevmode\vrule height 2pt depth -1.6pt width 23pt, {\em Metric measure
  spaces with {R}iemannian {R}icci curvature bounded from below}, Duke Math.
  J., 163 (2014), pp.~1405--1490.

\bibitem{DMGPS18}
{\sc S.~Di~Marino, N.~Gigli, E.~Pasqualetto, and E.~Soultanis}, {\em
  Infinitesimal {H}ilbertianity of locally \(\mathrm{CAT}(\kappa)\)-spaces}, J.
  Geom. Anal., 31 (2021), pp.~7621--7685.

\bibitem{ErikssonBiqueSoultanis21}
{\sc S.~Eriksson-Bique and E.~Soultanis}, {\em Curvewise characterizations of
  minimal upper gradients and the construction of a {S}obolev differential}.
\newblock Preprint, arXiv:2102.08097, 2021.

\bibitem{Gigli12}
{\sc N.~Gigli}, {\em On the differential structure of metric measure spaces and
  applications}, Mem. Amer. Math. Soc., 236 (2015), pp.~vi+91.

\bibitem{GigliNobili21}
{\sc N.~Gigli and F.~Nobili}, {\em A first-order condition for the independence
  on $p$ of weak gradients}.
\newblock Preprint, arXiv:2112.12849, 2021.

\bibitem{Heinonen07}
{\sc J.~Heinonen}, {\em Nonsmooth calculus}, Bull. Amer. Math. Soc., 44 (2007),
  pp.~163--232.

\bibitem{NPS21}
{\sc F.~Nobili, E.~Pasqualetto, and T.~Schultz}, {\em On master test plans for
  the space of \(\mathrm{BV}\) functions}.
\newblock To appear in Advances in Calculus of Variations, arXiv:2109.04980,
  2021.

\bibitem{Pas20}
{\sc E.~Pasqualetto}, {\em Testing the {S}obolev property with a single test
  plan}, Studia Mathematica,  (2022).
\newblock DOI: 10.4064/sm200630-24-8, published online: 3 March 2022.

\bibitem{Shanmugalingam00}
{\sc N.~Shanmugalingam}, {\em Newtonian spaces: an extension of {S}obolev
  spaces to metric measure spaces}, Rev. Mat. Iberoamericana, 16 (2000),
  pp.~243--279.

\end{thebibliography}
\end{document}